\documentclass{scrartcl}

\usepackage{graphicx}%
\usepackage{multirow}%
\usepackage{amsmath,amssymb,amsfonts}%
\usepackage{amsthm}%
\usepackage{mathrsfs}%
\usepackage[title]{appendix}%
\usepackage{xcolor}%
\usepackage{textcomp}%
\usepackage{manyfoot}%
\usepackage{booktabs}%
\usepackage{algorithm}%
\usepackage{algorithmicx}%
\usepackage{algpseudocode}%
\usepackage{listings}%
\RequirePackage{newunicodechar}
\RequirePackage{enumitem}
\RequirePackage{hyperref}
\RequirePackage{cleveref}
\RequirePackage{mathtools}

\newcommand{\ZZ}{\mathbb{Z}}

\newcommand{\NN}{\mathbb{N}}
\newcommand{\CC}{\mathbb{C}}
\newcommand{\RR}{\mathbb{R}}

\newcommand{\E}{\mathbb{E}}
\renewcommand{\P}{\mathbb{P}}
\newcommand{\Pl}[2][]{\P_{#1}\left(#2\right)}
\newcommand{\El}[2][]{\E_{#1}\left[#2\right]}
\DeclarePairedDelimiterXPP\Pk[1]{\P}{(}{)}{}{
  \providecommand\given{\nonscript\:\delimsize\vert\nonscript\:\mathopen{}}  
#1}
\DeclarePairedDelimiterXPP\Ek[1]{\E}[]{}{
  \providecommand\given{\nonscript\:\delimsize\vert\nonscript\:\mathopen{}}  
#1}

\newcommand{\eins}{\mathbf{1}}

\newcommand{\eps}{\varepsilon}
\newcommand{\fracc}[2]{{\left(\frac{#1}{#2}\right)}}

\newcommand\surject\twoheadrightarrow

\newcommand\del\partial
\newcommand\ohne\backslash
\newcommand\inv{^{-1}}

\DeclareMathOperator{\Var}{Var}
\DeclareMathOperator{\Cov}{Cov}

\renewcommand\l\left
\renewcommand\r\right
\newcommand\minv\wedge
\newcommand\maxv\vee

\DeclarePairedDelimiter\pars{(}{)}
\DeclarePairedDelimiter\kl{(}{)}
\DeclarePairedDelimiterX{\absnorm}[1]{\lVert}{\rVert_\infty}{#1}

\DeclarePairedDelimiter{\abs}{\lvert}{\rvert}

\DeclarePairedDelimiter\mg{\{}{\}}

\newcommand\Ksi\Xi

\newunicodechar{≤}{\le}
\newunicodechar{≥}{\ge}
\newunicodechar{µ}\mu
\newunicodechar{²}{^2}
\newunicodechar{ł}{\lambda}
\newunicodechar{−}{-}
\newunicodechar{∂}{\partial}
\newunicodechar{¯}{\overline}
\newunicodechar{∀}{\forall}
\newunicodechar{α}{\alpha}

\newcommand\poisv[1]{\widetilde{#1}}
\newcommand\patv[1]{\widehat{#1}}
\newcommand\triecal[1]{\mathcal{#1}} 
\newcommand\Alphabet{\mathcal A}
\newcommand\strings{\mathcal A^\ast}
\newcommand\trees{\mathfrak T}
\newcommand\triel[1][]{\poisv{\triecal T}_{#1\lambda}}

\newcommand\ttrien{\poisv{\triecal T}_{n}}
\newcommand\trien{\triecal T_n}
\newcommand\triek[1][k]{\triecal T_{#1}}
\newcommand\ptrien{\triecal P_n} 
\newcommand\ptriel{\poisv{\triecal P}_ł}
\newcommand\ptrie[1]{{\triecal P}_{#1}}
\newcommand\ptriek{\triecal P_k}

\newcommand\tphi{\patv\varphi}
\newcommand\tPhi{\patv\Phi}
\newcommand\dto{\overset d\longrightarrow}

\newcommand\dapprox{\overset d\approx}

\newcommand\hs{J}
\newcommand\nl[1][]{\nonumber\\&\phantom{#1}}

\DeclareMathOperator\pat{pat}
\DeclareMathOperator\lcp{lcp}

\newtheorem{theorem}{Theorem}
\newtheorem{proposition}[theorem]{Proposition}%
\newtheorem{lemma}[theorem]{Lemma}

\theoremstyle{remark}%
\newtheorem{remark}[theorem]{Remark}%

\theoremstyle{definition}
\title{Central Limit Theorems for Additive Functionals of Patricia Tries} 

\author{Jasper Ischebeck\thanks{Uppsala University, Sweden\\ \texttt{jasper.ischebeck@uu.se}}}

\begin{document}

\maketitle

\begin{abstract}
General additive functionals of patricia tries are studied asymptotically in a probabilistic model with independent, identically distributed letters from a finite alphabet. Asymptotic normality is shown after normalization together with asymptotic expansions of the moments.  There are two regimes depending on the algebraic structure of the letter probabilities, with and without oscillations in the expansion of moments.

As applications firstly the proportion of fringe trees of patricia tries with $k$ keys is studied, which is oscillating around $(1-\rho(k))/(2H)k(k-1)$ where $H$ denotes the source entropy and $\rho(k)$ is exponentially decreasing. The oscillations are identified explicitly. Secondly, the independence number of patricia tries and of tries is considered. The general results for additive functions also apply, where a leading constant is numerically approximated.

The results extend work of Janson \cite{janson} on tries by relating additive functionals on patricia tries to additive functionals on tries.
\end{abstract}


\section{Introduction}
Patricia tries are trees used in Computer Science as data structures to facilitate searching and sorting strings and in compression \cite{knuth}.
Many properties that relate to the performance to algorithms on patricia tries can be described by \textit{additive functionals}, such as path length and size. These are functions on trees that can be written as sums over all nodes (see \eqref{defi:phit} later).
These properties are typically studied in settings where the strings used to build the patricia tries arise from stochastic models.
We fix a finite alphabet $\Alphabet$ and a distribution $p$ on
$\Alphabet$ and sample infinite strings with independently identically distributed (i.i.d.) letters according to
$p$. For $a\in \Alphabet$, we write $p_a$ for the point mass of $p$ at $a$.
To avoid trivialities, we require $0 < p_a < 1$ for all $a\in\Alphabet$.

In this model, we show that when the number of strings in the patricia trie go to infinity, the values of additive functionals converge in distribution and with all moments against a normal distributed variable. The mean and variance are known to oscillate according to the algebraic structure of the letter probabilities $p_a$ and a first-order expansion including oscillations in terms of the contribution of each node is given.
For this, we require that the contribution of each node to the additive functional is sublinear in the size of its subtree.

Some important properties of patricia tries and the related tries are included in Knuth \cite{knuth} and since then, many other properties have been studied. By studying those properties one by one, one can get precise asymptotic expansions, like in
\cite{szpankowski:patrica_tries,jacquet_szpankowski_2015, bourdon, vallee}. Methods and results applying to multiple properties at once have been developed by Hwang, Fuchs and Zacharovas \cite{hwang:trievar} using analytic methods and recently by Janson \cite{janson} using more probabilistic methods.
Asymptotic results for small additive functionals have been found for many other random tree models, e.g.\ \cite{wagner:additive,janson:gw-fringe} for simply generated trees / conditioned Galton-Watson-trees, \cite{wagner:additive,Aldous1991AsymptoticFD} for recursive trees and binary search trees and \cite{janson} for tries, as mentioned before. See \cite[Section 3.1.2]{drmota:rt} for an overview.

After introducing patricia tries and tries, we show that there is a pullback of additive functionals on patricia tries to additive functionals on tries. Using this pullback, we show our main theorem, \Cref{zgs} by applying Theorem~3.9 from Janson \cite{janson} to this pullback. 
Then, we apply this theorem to the distribution of fringe trees, calculating the mean and variance of the number of fringe trees of a certain size. As a second application, we consider the independence number of the patricia trie and exemplify how to approximate mean and variance if they cannot be easily derived from the additive functional.

\section{Preliminaries}

The preliminaries are largely the same as for \cite{janson}, so we will try to keep this short and refer readers to \cite{janson} for a more detailed description.
For our fixed, finite alphabet $\Alphabet$ we define $\strings := \bigcup_{n=0}^\infty
\Alphabet^n$ as the set of finite strings over $\Alphabet$. The elements of strings are called \textit{characters}.
We write $\eps$ for the empty string and $\alpha\beta$ for the concatenation of strings $\alpha,\beta \in \strings$. We extend this notation to sets of strings in the usual manner, with $aT := \mg{at \mid t\in T}$ for $T\subseteq \strings,a\in \strings$.
As noted before, we have a probability measure $p$ on $\Alphabet$, with
point masses $p_a \in (0,1), a\in \Alphabet$.
For a string $\alpha = (a_1, a_2, \dots, a_n) \in \strings$ we define by extension $p_\alpha := p_{a_1}\dots p_{a_n}$.

This set $\strings$ can naturally be identified with the nodes of the infinite $\abs\Alphabet$-ary
tree, with a node $a\in \strings$ being an ancestor of a node $b\in \strings$ if and only if
$a$ is a prefix of $b$. We call subtrees of $\strings$ including the root
\textit{$\Alphabet$-ary trees} and associate them with their node set $T\subset \strings$.
For technical reasons, we consider also the empty graph $\emptyset$ as a tree.
Associating the trees with their node sets forms a bijection between $\Alphabet$-ary trees and subsets $T\subset \strings$ with the property that for all $a\in T$ all prefixes of $a$ are also in $T$.
We write $\trees$ for the set of $\Alphabet$-ary trees and let $\bullet := \mg{\eps}$ be the tree consisting only of the root.

For a set $T\subset \strings \cup \Alphabet^{\NN}$ of finite or infinite strings and a finite string $v\in \strings$, we define $T^v$ as the strings
starting with $v$ with $v$ removed. In exact terms, that means $T^v := \mg{\alpha\mid v\alpha\in T}$.
If $T$ is a tree and $v$ a node  of $T$, we call $T^v$ the \textit{fringe tree}. It is then the subtree of $T$ consisting of $v$ and its descendants.

Given a finite set $\mathfrak X \subseteq \strings \cup \Alphabet^{\NN}$ of 
(finite or infinite) strings
where no string is a prefix of another, the 
\textit{trie} $T(\mathfrak X)$ is a $\Alphabet$-ary tree defined as follows, see \cite[Section 6.1]{jacquet_szpankowski_2015} and \cite[Chapter 6.3]{knuth}:
The trie for the empty set is empty. The trie for a single string consists of only one leaf, the root. We say that this leaf \textit{stores} the string. The interesting case is the last: The trie for multiple strings consists of the root and
for every character $a\in \Alphabet$ this root has a child with the subtree being
a trie built from $\mathfrak X^a$, the strings starting with $a$. We are therefore splitting the
set of strings into subtrees by their first character. Formally, we have
$T(\emptyset) := \emptyset$, for $\abs{\mathfrak X} = 1$ we have $T(\mathfrak X) := \bullet$  and for $\abs{\mathfrak X} \ge 2$
\begin{equation}
    T(\mathfrak X) := 
        \mg{\eps} \cup \displaystyle\bigcup_{a\in\Alphabet} aT(\mathfrak X^a).
\label{defi:trie}
\end{equation}


In the \textit{patricia trie}, introduced by
Morrison \cite{patricia} in 1968, a node $v$ can additionally store a string $I_v$ as
the \textit{common prefix} of all strings that are stored in leaves under it.
We write $\lcp(T)$ for the longest common prefix of a set $T\subset \strings \cup \Alphabet^{\NN}$ of strings. The patricia trie for one or zero strings is the same as the trie. For multiple strings, the longest common prefix $\lcp(\mathfrak X)$ is saved
in an attribute $I$ of the root, and the splitting is done on the first
character that not all strings have in common. Thus, for every character $a\in A$, the root has
as subtree the patricia trie built from $\mathfrak X^{\lcp\kl{\mathfrak X}a}$, the strings starting with $\lcp(\mathfrak X)a$.
Formally, we define $\pat T(\emptyset)$ as $\emptyset$ and $\pat T(\mathfrak X) = \bullet$ for $\abs{\mathfrak X} = 1$ as for tries.
For $\abs{\mathfrak X} \ge 2$ we define
\begin{equation}
    \pat T(\mathfrak X) := 
        \mg{\eps} \cup \displaystyle\bigcup_{a\in\Alphabet}
        a\pat T(\mathfrak X^{\lcp(\mathfrak X)a}).
\end{equation}
Then, the common prefix $\lcp\kl{\mathfrak X}$ is stored in the attribute $I_\eps$ of the root, and for every $a\in\Alphabet$, $\alpha\in\strings$ the attribute $I_{a\alpha}$ will be the attribute $I_{\alpha}$ of $\pat T(\mathfrak X^{\lcp(\mathfrak X)a})$. While the patricia trie technically consists of the tree and the map
$v \mapsto I_v$ of the common prefixes, we usually regard it as tree
and ignore the additional structure when convenient.


We use $\pat T(\mathfrak X)$ as suggestive notation to put emphasis on the fact that the patricia trie can also be constructed from the trie.
For a trie $T$, the patricia trie $\pat(T)$ for the same strings
is constructed by iteratively merging 
every node $\nu\in\strings$ with exactly one child $\nu a$ for $a\in \Alphabet$ with its child and prepending the
character $a$ to the attribute $I_\nu$ of the merged node. Conversely,
by adding a parent with only one child for every character in $I_\nu$
for every node $\nu$ one can construct the trie from the patricia trie.
With simple induction one can show that both definitions actually are equivalent, so that we have $\pat(T(\mathfrak X)) = \pat T(\mathfrak X)$.

Let $\trien$ denote the random trie generated by $n\in \NN$ i.i.d.\ infinite strings, where each character is i.i.d.\ sampled from the distribution $p$ and $\ptrien = \pat(\trien)$ the patricia trie from the same strings.
We also consider a Poisson version, where we construct the trie
from a Poi($ł$)-distributed number $N_ł$ of strings for $ł \ge 0$, with $N_ł$ independent of the strings. We define $\triel := \triek[N_ł]$ and
$\ptriel := \ptrie{N_ł} = \pat\kl{\triel}$.

The properties that we study are represented as so-called additive
functionals. Given a function $\varphi: \trees \to \RR^n, n≥1$ with 
$\varphi(\varnothing) =0$, we define the
corresponding \textit{additive functional} $\Phi: \trees \to \RR^n$
as
\begin{equation}\label{defi:phit}
    \Phi(T) = \sum_{α\in \strings} \varphi(T^α)
    = \sum_{v\in T} \varphi(T^v).
\end{equation}
This can be written recursively as
\begin{equation}
    \Phi(T) = \varphi(T) + \sum_{a\in\Alphabet}\Phi(T^a);\quad \Phi(\varnothing)=0,
\end{equation}
which also shows that every function $\Phi: \trees \to \RR^n$ with $\Phi(\varnothing)=0$ can be represented as an additive functional.
The term ``additive functional'' for $\Phi$ is thus mainly defined by
its relation
to $\varphi$, which is called \textit{toll function} of $\Phi$.
For example, if we define 
\begin{equation}
    \varphi_i(T) := \eins\mg{T\text{ has more than 1 node}},
\end{equation}
then the corresponding additive functional $\Phi_i$ counts the internal
nodes of a tree. For a patricia trie $\pat T(\mathfrak X)$, we assume that $\Phi(\pat T(\mathfrak X))$ only depends on the patricia trie regarded as a $\Alphabet$-tree and ignores the additional attributes $I_v$.

We use the notation $\dto$ for convergence in distribution and
$\dapprox$ for approximation in distribution, i.e.\ for two sequences $X_n$ and  $Y_n$ of random variables on the same Borel space $S$, $X_n \dapprox Y_n$ holds if and only if for all bounded, continuous functions $f:\:S\to \RR$,
\begin{equation}\label{def:dapprox}
    \abs[\big]{\Ek{f(X_n)}- \Ek{f(Y_n)}} \longrightarrow 0 \qquad\text{as $n\to\infty$.}
\end{equation}
If $S=\RR$, we say that this approximation \emph{holds with all moments} if \eqref{def:dapprox} is also true for $f(x)=\abs{x}^r$, $r>0$ and $f(x) = x^s$, $s\in\NN$. Compare Appendix~B of \cite{janson}.
Furthermore, we write $H$ for the \textit{entropy} of the source $p$,
that is
\begin{equation}
    H := \sum_{a\in\Alphabet} p_a\log\kl*{\frac1{p_a}}.
\end{equation}
For complex $s\in \CC$ let
\begin{equation}\label{defi:rho}
    \rho(s) := \sum_{a\in\Alphabet}p_a^s.
\end{equation}
For natural $s\in \NN$, we may interpret
this as the probability for $s$ strings to start with the same character.

\section{Patricia trie and trie}
We now relate additive functionals on patricia tries to those on tries.
\begin{proposition}\label{defi:tphi}
    An additive functional $\Phi$ on patricia tries defines 
    an additive functional
    $
        \tPhi := \Phi\circ \pat
    $ 
    on tries
    by pullback with $\pat$. 
    The toll function
    $\tphi(T)$
of $\tPhi$ is given by
\begin{equation}
    \tphi(T) = \begin{cases}
        0 & T \text{'s root has exactly one child} \\
        \varphi(\pat(T)) & \text{else.}
    \end{cases}
\end{equation}
\end{proposition}
\begin{proof}
    Each node $\patv v$ of the patricia trie $\pat(T)$
    was created by compressing nodes in the trie. Associate
    $\patv v$ with the youngest node $v$ of the compressed nodes
    (the one that is a descendant of all the others).
    In this way, we have a bijection between the nodes of the
    patricia trie and the nodes in the trie which have not exactly
    one child.

    Compressing the nodes does not change the ancestor-descendant relationship, so
    the fringe tree of $\patv v$ consists of the compressed nodes
    of the fringe tree of $v$. Thus, $\pat(T)^{\patv v}$ is the same as  $\pat(T^v)$,
    except that in $\pat(T^v)$ there is no common prefix in the root $(I_{\eps} = \eps)$.
    Because the toll function $\varphi$ is not allowed to depend
    on the common prefixes, we nevertheless
    have $\varphi(\pat(T)^{\patv v}) = \tphi(T^v)$.
    Nodes with exactly one child are ignored by $\tphi$, so
    summing over all nodes yields the equation.
\end{proof}

Since then $\Phi(\ptrien) = \tPhi(\trien)$, all results about additive
functionals on tries can be extended to patricia tries by just exchanging
the toll function $\varphi$ with $\tphi$ if the conditions are still satisfied. There is an alternative way
to write $\tphi$: By defining
\begin{equation}\label{defi:phip}
    \varphi_p(T) := \eins\mg{T\text{ has not exactly one child.}},
\end{equation}
we have $\tphi(T) = \varphi_p(T)\varphi\kl*{\pat(T)}$ for all toll functions $\varphi$.

\section{Limit theorems}
We reformulate the main Theorem 3.9 of \cite{janson} for
patricia tries.
Let $\Phi$ be an additive functional on patricia tries. We call the
functional \textit{increasing} if for every two sets $\mathfrak X \subseteq \mathfrak Y \subset \Alphabet^{\NN}$ of strings we have 
that $\Phi\kl*{\pat\kl{T(\mathfrak X)}} \le \Phi\kl*{\pat\kl{T(\mathfrak Y)}}$.
In other words, $\Phi$ is increasing if and only if $\Phi\circ\pat\circ T = \tPhi \circ T$ is monotonous.
Note that this not the same as $\Phi$ being increasing as an additive functional
on tries as in \cite{janson}: In patricia tries, new leaves can also
grow out of edges, making functionals such as ``number of edges
that connect two nodes with more than 3 children'' not increasing.

Let $\Phi$ be an additive functional on patricia tries and $\tPhi$
the corresponding pullback on tries as defined in \Cref{defi:tphi},
with toll function $\tphi$. Let $\chi := \varphi(\bullet)$ be the value
every leaf contributes to $\Phi$. The leaves have to be treated as a special case in the analysis of patricia tries (and tries) because their behavior is
substantially different from internal nodes. However, their number is
simply $n$ in $\ptrien$ and $N_ł$ in $\ptriel$, and we use $\chi$ only
to accommodate for this fact.

To express the asymptotics for mean and variance, we need the following functions, where the terms are much simpler for $\chi=0$:
\begin{align}
    \label{defi:fe}
    f_E(ł) &:= \El{\tphi(\triel)} - \chi ł e^{-ł}
    \\ \nonumber
    f_V(ł) &:= 2\Cov\pars*{\tphi(\triel), \tPhi(\triel)}
        - \Var\kl*{\tphi(\triel)}
    \nonumber\\ 
    &\phantom{:= }
        + 2\chi ł e^{-ł}(\E\tPhi(\triel)-\E\tphi(\triel))
        - \chi^2ł e^{-ł}(1-ł e^{-ł})
        \label{defi:fv}
    \\ \label{defi:fc}
    f_C(ł) &:= \Cov\pars*{\tphi(\triel), N_ł} + \chi ł(ł-1)e^{-ł}.
\end{align}
In the convenient case of $\chi=0$, this simplifies to
\begin{align}
    \label{defi:fe:simp}
    f_E(ł) &= \Ek*{\tphi(\triel)}
    \\
    f_V(ł) &= 2\Cov\pars*{\tphi(\triel), \tPhi(\triel)}
        - \Var\kl*{\tphi(\triel)}
    \label{defi:fv:simp}
    \\
    f_C(ł) &= \Cov\pars*{\tphi(\triel), N_ł}.
    \label{defi:fc:simp}
\end{align}
Now for $X=E,V,C$ we define the Mellin transform of $f_X$ as
\begin{equation}
    f_X^\ast(s) := \int_0^\infty t^{s-1}f_X(t)\mathrm dt,
\end{equation}
for $s\in \CC$ where the integral converges absolutely.

For certain distributions $p$, there are typically oscillations in
mean and variance of additive functionals. These occur if there
is a number $d_p$, such that
\begin{equation}\label{defi:fxs}
    \mg{\log p_a \mid a \in \Alphabet} \subseteq d_p\ZZ.
\end{equation}
We then define $d_p$ to be the biggest such number. If there is no
such number, we define $d_p := 0$. An alternative characterization is to define $d_p$ as the infimum of positive elements of the additive subgroup
generated by $\mg{\log p_a \mid a \in \Alphabet}$.

If $d_p=0$, there are no oscillations, and we define
\begin{equation}\label{defi:psix0}
    \psi_X(t) := f_X^\ast(-1)
\end{equation}
as constant functions. 
If $d_p > 0$, we define $\psi_X$ as bounded, continuous, $d_p$-periodic functions with the Fourier series
\begin{equation}\label{defi:psix}
    \psi_X(t) \sim \sum_{m=-\infty}^\infty f_X^\ast\pars*{-1-\frac{2\pi m}{d_p}i}
    e^{{2\pi i mt}/{d_p}}.
\end{equation}
Moreover, if $X=E$ or if $f'_X(\lambda) = O(\lambda^{-\eps_1})$ as $ł\to\infty$
for some $\eps_1>0$, \cite[Theorem 3.1]{janson} shows that the Fourier series \eqref{defi:psix} converges
absolutely, and thus ``$\sim$'' may be replaced by ``$=$'' in \eqref{defi:psix}.
\cite[Lemma 3.6]{janson} further shows that $\psi_C(t) = \psi_E(t) + \psi_E'(t)$ for all
values of $t$, so that one actually just needs to compute $\psi_E$ and $\psi_V$.
Typically, the values $f_X^\ast\pars*{-1-\tfrac{2\pi m}{d_p}i}$
are small for $m \ne 0$, so we can still consider the mean term
$f_X^\ast(-1)$ as ``average asymptotic value''.

With these values and definitions, we can finally state our version of \cite[Theorem 3.9]{janson}.

\begin{theorem}\label{zgs}
    Let $\varphi_+, \varphi_-$ be bounded toll functions on patricia tries so that their additive functionals $\Phi_+, \Phi_-$ are increasing.
    Then, $\Phi := \Phi_+-\Phi_-$ is also an additive functional on  patricia tries. Let $\tphi$, $\tPhi$ be toll function and additive functional
    on tries corresponding to $\Phi$ as in \Cref{defi:tphi}. Then, with the functions $f_X^\ast, \psi_X$ defined in \eqref{defi:fe}-\eqref{defi:psix}:
\begin{enumerate}[label=\roman*)]
    \item As $ł\to\infty$ and $n\to\infty$,
    \begin{align}
        \frac{\Phi(\ptriel) - \E[\Phi(\ptriel)]}{\sqrt ł} 
        &\dapprox
        \mathcal N(0, \poisv\sigma^2(ł)) 
        ,&
        \label{conv:sym:n}
        \frac{\Phi(\ptrien) - \E[\Phi(\ptrien)]}{\sqrt n} 
        &\dapprox
        \mathcal N(0, \sigma^2(n))
    \end{align}
    with all moments, where
    \begin{align}
        \label{ass:sigma}
        \poisv\sigma^2(ł) &= \chi^2 + H\inv \psi_V(\log ł) \\
        \label{ass:hsigma}
        \sigma^2(n) &= H\inv \psi_V(\log n)
        - H^{-2}\psi_C(\log n)^2
        - 2\chi H\inv \psi_C(\log n).
    \end{align}
    \item The expected values satisfy
    \begin{equation}\label{zgs:exp}
        \E[\Phi(\ptrien)] - \E[\Phi(\poisv{\triecal P}_n)] \in o\pars*{\sqrt{n}},
    \end{equation}
    and we may thus replace $\E[\Phi(\ptrien)]$ with $\E[\Phi(\poisv{\triecal P}_n)]$
    in \eqref{conv:sym:n}.
    \item \label{zgs:varn} If
    $
        \Var\Phi(\ptrien) \in \Omega(n),
    $
    then
    \begin{align}
        \frac{\Phi(\ptriel) - \E[\Phi(\ptriel)]}
            {\sqrt{\Var\Phi(\ptriel)}} &\dto
        \mathcal N(0, 1) 
        ,&
        \label{conv:var:n}
        \frac{\Phi(\ptrien) - \E[\Phi(\ptrien)]}
            {\sqrt{\Var\Phi(\ptrien)}} &\dto
        \mathcal N(0, 1)
    \end{align}
    with all moments for all values of $d_p$.
    \item The expected values $\E[\Phi(\trien)],\E[\Phi(\triel)]$
        satisfy
    \begin{align}
        \frac{\E[\Phi(\triel)]}{ł} &= \frac{\psi_E(\log ł)}H +o(1)
        ,& \label{conv:ephi:n}
        \frac{\E[\Phi(\ttrien)]}{n} &= \frac{\psi_E(\log n)}H +o(1).
    \end{align}
\end{enumerate}
\end{theorem}

\begin{remark}
    The assumption of boundedness on $\phi_+, \phi_-$ can, as in \cite{janson}, be weakened to show that for some $\eps>0$,
    \begin{equation}
        \Ek{\phi_\pm(\triel)}  = \mathrm O(\lambda^{1-\eps}),\qquad \Var\kl{\phi_\pm(\triel)} = \mathrm O(\lambda^{1-\eps})
    \end{equation}
    and that there exists some $r>2$ such that
    \begin{equation}
        \E\abs*{\phi_\pm(\triel) - \E\phi_\pm(\triel)}^r = \mathrm O(\lambda^{r/2}).
    \end{equation}
    The convergence \eqref{conv:var:n} then only holds for moments $s<r$, cp. Theorem~5.8 in \cite{janson}.
\end{remark}

\begin{remark}
    \Cref{zgs} also applies to linear combinations of additive functionals which meet the conditions of \Cref{zgs}, so one can use the Cramér–Wold device to show a multivariate version of \Cref{zgs} (cp.\ the proof of Theorem~5.6 in \cite{janson}).
\end{remark}

From his version of \Cref{zgs} Janson derives a weak law of large numbers, while
leaving as an open question if one also had almost sure (a.s.) convergence. Using
a proof technique for the strong law of large numbers for sums of i.i.d.\ 
random variables with a finite fourth absolute moment,
we can show a.s.\ convergence. We have the following strong law
of large numbers for patricia tries. The same proof can be applied to tries, answering the open question.

\begin{theorem}\label{stggrz}
    Let $\varphi$ be a bounded toll function on patricia tries as in \Cref{zgs}.
    Then, as $n\to\infty$,
    \begin{align}\label{stggrz:n}
        \frac{\Phi(\ptrien)}{n} - H\inv \psi_E(\log(n)) - \chi 
        \longrightarrow 0\quad\textnormal{ a.s.}
    \end{align}
    In particular, if $d_p=0$, then, as $n\to\infty$,
    \begin{equation}\label{stggrz:ass}
        \frac{\Phi(\ptrien)}n  \longrightarrow H\inv f_E^\ast(-1) + \chi \quad\textnormal{ a.s.}
    \end{equation}
\end{theorem}

\begin{proof}
    Note that $\E[\Phi(\ptrien)]$ is asymptotically 
    $H\inv \psi_E(\log ł) + \chi$ according to \eqref{conv:ephi:n}. So \eqref{stggrz:n}
    is equivalent to showing that 
    $n\inv Z_n := n\inv(\Phi(\ptrien) - \E\Phi(\ptrien)) \to 0$
    a.s.
    Writing the quantifiers of ``not converging'' out, we have
    \begin{equation}
        \mg*{\frac{Z_n}n \not\to 0} = \bigcup_{m\in\NN}\bigcap_{N=1}^\infty 
        \bigcup_{n=N}^\infty \mg*{\abs*{\frac{Z_n}n}>\frac1m}
        \label{stggrz:proof:quant}
        = \bigcup_{m\in\NN} \limsup_{n\to \infty}\mg*{\abs*{\frac{Z_n}n}>\frac1m}.
    \end{equation}
    Since we have convergence of all moments of $Z_n/\sqrt n$ 
    in \eqref{conv:sym:n}, we have
    $\E\abs{Z_n/n}^4 ≤ C/n^2$ for a constant $C>0$ and all $n≥1$.
    Markov's inequality implies
    \begin{align*}
        &\phantom= \sum_{n=1}^\infty \Pl{\abs*{\frac{Z_n}n}>\frac1m}
        = \sum_{n=1}^\infty \Pl{\abs*{\frac{Z_n}n}^4>\frac1{m^4}}
        \le \sum_{n=1}^\infty \frac{m^4C}{n^2} < \infty,
    \end{align*}
    hence the Lemma of Borel-Cantelli yields
    \[
        \Pl{\limsup_{n\to \infty}\mg*{\abs*{\frac{Z_n}n}>\frac1m}}=0
    \]
    for all $m≥1$. In view of \eqref{stggrz:proof:quant} subadditivity
    implies the assertion.
\end{proof}

\section{Applications}

With \cref{zgs}, we can calculate the asymptotic distribution of random
fringe trees in patricia tries. Furthermore, we show another application, the independence number, to demonstrate how to use this theorem without a closed formula for the functions $f_E, f_V$.

\subsection{Fringe patricia tries}\label{sec:fsize}

Our main application is the study of random fringes trees $\ptrien^\ast$ of patricia tries.
We want to first study the size $\abs{\ptrien^\ast}_e$,
as measured in number of leaves or equivalently strings.
We count the number of subtrees of size $k≥1$ with the functional
\begin{equation}
    \varphi_k(T) := \eins\mg{|T|_e=k}.
\end{equation}
Note that subtrees of size $1$ are the leaves, so 
$\varphi_1 = \varphi_\bullet = \tphi_\bullet$. This case behaves differently from the others and is trivial,
so we will first consider only $k≥2$.

While $\Phi_k$ is not increasing, it can be written as 
$\Phi_{≥k}-\Phi_{≥k+1}$, where $\varphi_{≥k} := \eins\mg{|T|_e≥k}$
is bounded and $\Phi_{≥k}$ is increasing, so we can apply
\Cref{zgs} and \Cref{stggrz}.

We present two ways to get results about fringe patricia tries. One way is to directly calculate $f_{E,k}, f_{E,k}^\ast$ etc.
Another way is to directly link fringe trees on tries with those on patricia tries. For those, we first have to understand what differs from tries: the common prefixes.

\begin{lemma}\label{lemma:pt}
    Let $T$ be a fixed $\Alphabet$-nary tree where no node has outdegree 1 and $k := |T|_e$.
    Let $T_e \subset \strings$ be its leaves and $T_i$ be its internal nodes.
    The probability $p_T := \P(\ptriek = T)$ of a random patricia trie of size
    $k$ to be $T$ is given by
    \begin{equation}
        p_T = k!\prod_{v \in T_e} p_v \prod_{w\in T_i} \frac1{1-\rho(|T^w|_e)}.
    \end{equation}
    Conditioned on that the tree structure of $\ptriek$ is given by $T$, the common prefixes per node are
    independent. In an internal node $v\in T_i$, their distribution is given by
    $q_{|T^v|_e}$, with $q_i$, $i≥2$ defined as
    \begin{equation}
        q_i(\mg{α}) := p_α^i (1-\rho(i));\;α\in\strings.
    \end{equation}
    Thus, the length $\abs I_v$ is $\mathrm{Geom}_0\kl*{1-\rho\kl*{\abs{T^v}_e}}$-distributed, where we use $\mathrm{Geom}_0$ to refer to the geometric distribution with support $\NN_0$ and $\mathrm{Geom}_1$ to the one with support $\NN^+$.
\end{lemma}

We will prove \Cref{lemma:pt} at the end of this section. The trie can be reconstructed from the tree structure and the common prefixes of the patricia trie.
As $\pat$ is compressing chains of nodes into the common prefixes, the inverted process basically expands those common prefixes into a chain of nodes.
Every fringe patricia trie $T^v$ of size $k$ then corresponds to $\abs{I_v} + 1$ many fringe tries of size $k$. By \Cref{lemma:pt}, $\Phi_k(\trien)$ is thus distributed as the sum
of $\Phi_k(\ptrien)$ many $\mathrm{Geom}_1(1-\rho(k))$-distributed, independent random variables.
This immediately links mean and variance of the number of
fringe trees of size $k$:
\begin{align}
    \Ek{\Phi_k(\trien)} &= \frac1{1-\rho(k)}\Ek{\Phi_k(\ptrien)}
    \\ \Var\kl[\big]{\Phi_k(\trien)} &= 
    \Ek[\big]{\Var\kl{\Phi_k(\trien) \mid \Phi_k(\ptrien)}}
    + \Var\kl[\big]{\Ek{\Phi_k(\trien) \mid \Phi_k(\ptrien)}}
    \nl=\frac{\rho(k)}{\kl{1-\rho(k)}^2} \Ek{\Phi_k(\ptrien)}
    + \frac1{\kl{1-\rho(k)}^2}\Var\kl{\Phi_k(\ptrien)}.
    \label{link:var}
\end{align}
For fixed $k$, mean and variance for patricia tries are hence just linear combinations of mean and variance for tries, so this will also hold for $f_{E,k}, f_{V,k}$
and their Mellin transforms $f_{E,k}^\ast, f_{V,k}^\ast$.

Nevertheless, we show how to calculate them directly, starting with the mean.
From \Cref{defi:tphi} and \eqref{defi:phip} we have $\tphi_k(\triel) = \eins\mg{N_ł = k}\varphi_p(\triel)$,
so the mean function is given by
\begin{align}
    f_{E,k}(ł) &= \El{\tphi_k(\triel)}
    = \P(N_ł = k)\E[\varphi_p(\triel)\mid N_ł=k]
    \nl = \frac{ł^k}{k!}e^{-ł}
        \pars*{1-\sum_{a\in\Alphabet}p_a^k}
        = \frac{ł^k}{k!}e^{-ł}
        \pars*{1-\rho(k)}
    \label{fekl}
\end{align}
with $\rho(k)$ as defined in \eqref{defi:rho}.
By the definition of the Gamma function $\Gamma$, the Mellin transform is given by
\begin{align}
    f^\ast_{E,k}(s) &= (1-\rho(k))\frac{\Gamma(k+s)}{k!},
    &\label{fekm1}
    f^\ast_{E,k}(-1) &= \frac{1-\rho(k)}{k(k-1)}.
\end{align}
Now to the variance. The function $f_{V,k}$ is given by
\begin{align}
    f_{V,k}(ł) &:= 2\Cov\pars*{\tphi_k(\triel), \tPhi_k(\triel)}
        - \Var\kl*{\tphi(\triel)}.
\end{align}
Note that if $\tphi_k(\triel)$ is 1, that is if the root has more than one child and there are $k$ strings in the trie, all other nodes must have less than $k$ strings. So then $\tPhi_k(\triel) = 1$. One then uses monotonous convergence to show
\begin{align}
    f_{V,k}(ł) &= \Ek{\tphi_k(\triel)} - 2\Ek{\tphi_k(\triel)}\Ek{\tPhi_k(\triel)} + \Ek{\tphi_k(\triel)}^2
    \nl
    = f_{E,k}(ł) - 2\sum_{α\in \strings}f_{E,k}(ł)\Ek{\tphi_k(\triel^α)} + f_{E,k}^2(ł)
    \nl
    = f_{E,k}(ł) - 2\sum_{α\in \strings} f_{E,k}(ł)f_{E,k}(p_α ł) + f_{E,k}^2(ł).
    \label{fvkl:sum}
\end{align}
The last equality is because $\triel^α$ is the trie of the strings starting with $α$ (at least if it exists), a fact which can be directly deduced from the definition in \eqref{defi:trie}. Since number of strings starting with $\alpha$ is Poi$(łp_α)$-distributed, $\tphi_k(\triel^α)$ is distributed as $\tphi_k(\triel[p_α])$.
Note that the last term is the term in the sum for $\alpha = \eps$.
Sums like in \eqref{fvkl:sum} show up often in the analysis of $f_V$, so Janson defines the shorthand ${\sum_α}^\ast := \sum_{α\in\strings} + \sum_{α\in\strings\ohne\mg\eps}$, which summarizes our last two terms. We can then substitute $f_{E,k}$ in using \eqref{fekl}:
\begin{align}
    f_{V,k}(ł) &= \frac{1-\rho(k)}{k!}ł^k e^{-ł} - {\sum_α}^\ast
    \fracc{1-\rho(k)}{k!}^2 ł^{2k}p_α^k e^{-ł(1+p_α)}.
    \label{fvkl}
\end{align}
The Mellin transform, which interchanges with the sum because of monotone convergence, can then be expressed using the Gamma function.
\begin{align}
    \label{fvks}
    f_{V,k}^\ast(s) &= \frac{1-\rho(k)}{k!}\Gamma(k+s)
    - {\sum_α}^\ast
    \fracc{1-\rho(k)}{k!}^2 p_α^k\Gamma(s+2k) (1+p_α)^{-s-2k}
    \\\label{fvkm1}
    f_{V,k}^\ast(-1) &= 
    \frac{1-\rho(k)}{k(k-1)} - 
    (2k-2)!\fracc{1-\rho(k)}{k!}^2
    {\sum_α}^\ast
    \frac{p_α^k}{(1+p_α)^{2k-1}}.
\end{align}

\begin{theorem}\label{satz:fk}
    \Cref{zgs} and \ref{stggrz} apply to the number $\Phi_k(\ptrien)$ of fringe trees with $k\ge 2$ strings in a patricia trie of $n$ strings.
    Especially, we have
    \begin{equation}
        \frac{\Phi_k(\ptrien)-\E\Phi_k(\ptrien)}{\sqrt{n}}
        \dapprox \mathcal N(0,\sigma^2(\log n))
    \end{equation}
    with convergence of all [absolute] moments. The variance $\sigma^2(\log n)$ is bounded and periodic.
    In the $d_p=0$ case, we have convergence
    \begin{equation}
        \frac1n \Phi_k(\ptrien) \longrightarrow \frac{1-\rho(k)}{Hk(k-1)}
    \end{equation}
    a.s.\ and with all moments and the term  $\sigma^2(\log n)$ is constant and given by
    \begin{align}
        \sigma^2(\log n) = \frac{1-\rho(k)}{Hk(k-1)} - 
        \frac{(2k-2)!}{H}\fracc{1-\rho(k)}{k!}^2
        {\sum_α}^\ast
        \frac{p_α^k}{(1+p_α)^{2k-1}}.
    \end{align}
    If further also $\abs\Alphabet = 2$, this means that
    \begin{equation}\label{pk_a2}
        \Pk{\abs{\ptrien^\ast}_e = k \given \ptrien} \longrightarrow
        \frac{1-\rho(k)}{2Hk(k-1)}
    \end{equation}
    a.s.\ and with all moments. 
    In the $d_p>0$ case, small oscillations around the above values occur, see \Cref{zgs}.
\end{theorem}
\begin{proof}
The theorem is just a straightforward application of \Cref{zgs}, with
the values for $f_{E,k}^\ast(-1)$ and $f_{V,k}^\ast(-1)$ from \eqref{fekm1} respective \eqref{fvkm1}. The statement for $\abs\Alphabet = 2$ follows from the fact that for $\abs\Alphabet = 2$ all internal nodes have the same outdegree of 2, so that $\abs\ptrien = 2n-1$ deterministically and therefore \begin{equation}
    \Pk{\abs{\ptrien^\ast}_e = k \given \ptrien} = \frac1{2n-1}
    \Ek{\Phi_k(\ptrien) \given \ptrien} \quad\text{a.s.}
\end{equation}
\end{proof}
\begin{remark}
    It is possible to get results like \eqref{pk_a2} also for $\abs\Alphabet > 2$. Theorem 4.7 from \cite{janson}, which gives the asymptotic normality and asymptotics for $\Pk{\abs{\trien^\ast}_e = k \given \trien} = \tfrac{\Phi_k(\trien)}{n+\Phi_i(\trien)}$, holds similarly also for patricia tries, with
    just the functionals for size and fringe tree count replaced by their patricia trie counterparts. The quantity
    \begin{equation}\label{defi:hstrich}
        \hs := \sum_{a\in\Alphabet} (1-p_a)\log\kl*{\frac1{1-p_a}},
    \end{equation}
    describes the expected size of a patricia trie, which is
    (ignoring oscillations) $\tfrac{J+H}{H}n+o(n)$, see \cite{bourdon}.
    The distribution of fringe trees is then asymptotically
    \begin{equation}
        \Pk{\abs{\ptrien^\ast}_e = k \given \ptrien} =
            \frac{1-\rho(k)}{(J+H)k(k-1)} + \text{oscillations} + o(1).
    \end{equation}
    See \cite{master} for the detailed terms of the oscillations and the variance of $\Pk{\abs{\ptrien^\ast}_e = k \given \ptrien}$ using this analysis.
\end{remark}

Compare this to the asymptotic distribution of fringe tries, which oscillates around $\tfrac{1}{(1+H)k(k-1)}$, see \cite[Theorem 4.4]{janson}.
In the trie, the distribution only depends on the entropy $H$, whereas
in the patricia trie there is an additional dependency on the source in form of $\rho$.
Other known asymptotic distribution of fringe trees include the random recursive tree, which has a limit of $\tfrac 1{k(k+1)}$\cite{Aldous1991AsymptoticFD}.
While these probabilities decay with $k^{-2}$ for large $k$, the probabilities for conditioned Galton-Watson decay more slowly, with order $k^{-3/2}$, see \cite{Aldous1991AsymptoticFD}.

\begin{remark}
    For a fixed $\Alphabet$-ary tree $T$, one can count the number $\Phi_T$ of fringe trees equal to $T$.
    Because the probability of $\ptriel$ to be $T$ is the probability of the $\ptriel$ having $k$ strings and the patricia trie of the $k$ strings being $T$, the mean function is given by
    $f_{E,T}(ł) = f_{E,k}(ł)\Pk{\ptriek = T}$.
    The rest of the analysis is largely the same as for $\Phi_k$ and one can get an
    equivalent of \Cref{satz:fk} for $\Phi_T$ and the tree structure of $\ptrien^\ast$.
\end{remark}

We finish this section with the proof of \Cref{lemma:pt}.
\begin{proof}[Proof of \Cref{lemma:pt}]
    In a trie there is exactly one string that has
    a leaf as a prefix. So for a valid trie $T'$, the probability of
    a random trie of size $k$ to be $T'$ is 
    $p_{T'} := k!\prod_{v\in T_e}p_v$.
    Having no node of outdegree 1 makes $T$ a valid patricia trie. Write $\pat\inv(T)$ for the set of all tries $T'$ such that $\pat T'$ has the tree structure $T$.
    The probability $p_T$ is the sum of the $p_{T'}$ of the
    tries $T'\in \pat\inv(T)$ that
    are given by assigning common prefixes to the inner nodes.
    Fixing the common prefixes $I_v \in \strings$ for every inner node $v\in T_i$
    gives us a unique trie $T'$. Let $w\in T$ be a leaf in the patricia trie, which corresponds to a leaf $\patv w\in T'$. This leaf $\patv w$, when seen as a string, consists of the characters in $w$ interleaved with all common prefixes $I_v$ of the true ancestors $v\in T$ of $w$. To be precise, if $w=a_1\cdots a_r$, then
    $\patv w = I_\eps a_1 I_{a_1} a_2 I_{a_1a_2} a_3 \cdots I_{a_1\dots a_{r-1}} a_r$.
    Thus, writing $v \le w$ for ``$v$ is ancestor of $w$'', we have $p_{\patv w} = p_w \prod_{v < w} p_{I_v}$.
    Multiplying $p_w$ over all leaves, $p_{I_v}$ gets multiplied as often as there are leaves in $T^v$.
    So the probability $p_{T'}$ is given by
    \begin{align}
        p_{T'}
        &=
        k!
        \kl*{\prod_{w\in T_e}p_w}
        \kl*{\prod_{v\in T_i} p_{I_v}^{|T^v|_e}}
        \nl
        = k!
        \kl*{\prod_{w\in T_e}p_w}
        \kl*{\prod_{v\in T_i} \frac1{1-\rho(|T^v|_e)}}
        \kl*{\prod_{v\in T_i} q_{|T^v|_e}(\mg{I_v})}.
    \end{align}
    Note that this term is the product of the proposed term for $p_T$ and the proposed point masses of $q_{\abs{T^v}_e}$.
    So if we show that the $q_i$ for $i≥2$ really are
    distributions and sum up to 1, this already shows the term for $p_T$, $q_i$ and the independence given the tree structure.
    Summing over strings of length $n \ge 0$, we have
    \begin{align*}
        \sum_{α\in\Alphabet^n} q_i(\mg{α})
        &= (1-\rho(i)) \sum_{α\in\Alphabet^n} p_{α}^i
        = (1-\rho(i)) \kl*{\sum_{a\in\Alphabet}p_a^i}^n
        = (1-\rho(i))\rho(i)^n.
    \end{align*}
    We thus showed that the lengths are really $\text{Geom}_0(1-\rho(i))$, which makes $q_i$ indeed a well-defined distribution since $\rho(i) < 1$ for $i>1$.
\end{proof}

\subsection{Independence Number}
Given a graph $G = (V,E)$, an \textit{independent set} $I\subseteq V$ is a set
of nodes such that there are no edges between two nodes in $I$.
The \textit{independence number} $\alpha(G)$ of a tree $G$ is the maximal
cardinality of an independent set.

While an algorithmically challenging problem on general graphs, on a tree $T$, the independence number $\alpha(T)$ can be calculated recursively, see \cite{indnum}, also for references on the asymptotic properties of the independence number of other random tree models.
The independence number only grows by at most 1 after adding
a node. For the toll function $\varphi_\alpha$ of $\alpha$ this means
that $\varphi_\alpha(T) := \alpha(T) - \alpha(T\ohne \mg \eps) \in \mg{0,1}$. If $\varphi_\alpha(T) = 0$, the largest independent set of $T$ is as large as the largest one of $T \ohne\mg\eps$. Therefore, we can find a maximal independence set excluding the root. If $\varphi_\alpha(T) = 1$,
the largest independent set of $T$ is not an independent set of $T\ohne\mg\eps$, therefore every maximal independent set must include the root $\eps.$ We call such nodes $\beta\in T$ with $\varphi_\alpha(T^\beta) = 1$ \textit{essential}.

An independent set containing the root cannot contain any of its children $b\in \Alphabet$. Hence, it is upper bounded by
\begin{equation}
    1 +\sum_{b\in \Alphabet} \alpha(T^b\ohne\mg b)
    = 1 + \alpha(T\ohne\mg\eps) - \sum_{b\in\Alphabet}\varphi_\alpha(T^b).
\end{equation}
Furthermore, an independent set not containing the root has maximal cardinality $\alpha(T\ohne\mg\eps)$. Thus, we have
\begin{equation}\label{defi:falpha}
    \varphi_\alpha(T) = \max\mg[\Big]{0, 1  - \sum_{b\in\Alphabet}\varphi_\alpha(T^b)};
\end{equation}
a node is exactly then essential
if none of its children are. Another related quantity is the \textit{matching number}, the maximal cardinality of a matching, which is just
given by $\abs T-\alpha(T)$ for trees.

The toll function $\varphi_\alpha$ is bounded and $\alpha$ is increasing. We can thus use \Cref{zgs} to get following theorem:

\begin{theorem}\label{indnum:sammelsatz}
    \Cref{zgs} applies to the independence number $\alpha(\ptrien)$ of a random patricia trie. Furthermore, $\Var\alpha(\ptrien) = \Omega(n)$ holds, and
    \begin{equation}\label{indnum:zgs}
        \frac{\alpha(\ptrien) - \Ek{\alpha\kl{\ptrien}}}
        {\sqrt{\Var\kl[\big]{\alpha(\ptrien)}}} \dto \mathcal N(0,1)
    \end{equation}
    with all moments. 
    If $d_p = 0$, the expectation is asymptotically
    \begin{equation}
        \frac{\Ek{\alpha(\ptrien)}}n = 
            1 + H\inv f_{E,\alpha}^\ast(-1)  + o(1)
    \end{equation}
    and if $d_p>0$, it is
    \begin{equation}\label{indnum:expansiond0}
        \frac{\Ek{\alpha(\ptrien)}}n = 
            1 + \frac1H \sum_{m=-\infty}^\infty f_{E,\alpha}^\ast\kl*{
                -1-\frac{2\pi mi}{d_p}
            }e^{2\pi im\log n/d_p}  + o(1).
    \end{equation}
    The values of $f_{E,\alpha}^\ast$ in the asymptotic expansions can be calculated to arbitrary precision. For example in the binary, symmetric case,
    the expected ratio 
    $\tfrac{\Ek{\alpha(\ptrien)}}{2n}$ of essential nodes is oscillating around
    \begin{equation}\label{indnum:expansiondne0}
        \frac 12+ \frac1{2H} f_{E,\alpha}^\ast(-1)
        \in (0.60225, 0.60316).
    \end{equation}
\end{theorem}

\begin{proof}
\eqref{indnum:zgs} follows from \Cref{zgs}\ref{zgs:varn}
if we show that $\Var\alpha(\ptrien) = \Omega(n)$. This can be
done by an extension of \cite[Lemma 3.14]{janson} as shown in \cite{master}. A rough sketch is as follows: We condition on all nodes but those included in fringe trees of size 5 (``5-fringes'') and furthermore on if the roots of the 5-fringes are essential. This already fixes the essentiality of all nodes outside the 5-fringes. 5-fringes with non-essential root can either have 5 or 6 essential nodes, so the conditioned variance is not zero. There are by \Cref{sec:fsize} $\Theta(n)$ many 5-fringes, so we have $\Var\alpha(\ptrien) ≥ \Ek{\Phi_5(\ptrien)}\Ek*{
    \Var\kl{\alpha(\ptrie5)\given \varphi_\alpha(\ptrie5)}} =  \Theta(n)$.

The asymptotic expansions \eqref{indnum:expansiond0}, \eqref{indnum:expansiondne0} are the definitions of $\psi_{E,\alpha}$ from \eqref{defi:psix0} and \eqref{defi:psix} substituted into \eqref{zgs:exp}.
Calculating the mean and variance is hindered by the recursive nature of \eqref{defi:falpha}, but we can estimate the asymptotic mean to arbitrary precision by counting only essential nodes with small fringe trees.

We calculate $\alpha_n := \E[\varphi_α(\ptrien)]$ for $0≤n≤N$
with $N\in\NN$.
Then $\El{α(\ptrien)}$ can be bounded by
\begin{equation}
    0 ≤ \El{α(\ptrien)} - \pars[\bigg]{n + \sum_{k=2}^N \E[\varphi_α(\ptriek)] 
    \El{\Phi_k(\ptrien)}} ≤ \El{\Phi_{>N}(\ptrien)}.
\end{equation}

We can now use the simultaneous convergence of $\Phi_i$ and $\Phi_k, 2≤k≤N$ from \Cref{sec:fsize}.
For $\Phi_{>N}$ we have
$\Phi_{>N} = \Phi_{i} - \sum_{k=2}^N \Phi_k$ and the fact that the limits
of $\E[\Phi_k], k≥2$ sum up to the limit of $\E[\Phi_i]$ and obtain in the $d_p=0$ case:
\begin{align}
    0 ≤ \lim_{n\to\infty}\frac{\El{α(\ptrien)}}n - \pars[\bigg]{
        1 + \sum_{k=2}^N \frac{(1-\rho(k))\alpha_k}{k(k-1)H}
    }
    &≤ \sum_{k=N+1}^\infty \frac{1-\rho(k)}{k(k-1)H}
    \nonumber\\\label{alpha:bound}
    &≤ \frac{1}{NH}.
\end{align}
In the periodic $d_p>0$ case, the same can be done for each value of $f_{E,\alpha}^\ast$.

The values $\alpha_k$ can be calculated by conditioning on the first step; this gives a recursive equation. For the symmetric, binary case, this is given by
\begin{equation}
    \alpha_n = \prod_{k=1}^{n-1}\binom nk 
        \frac{(1-\alpha_k)(1-\alpha_{n-k})}{2^n-2}
\end{equation}
for $n≥2$ and $\alpha_1 = 1, \alpha_0=0$. See \Cref{fig:alpha} for
a plot of the first values of $\alpha$.
Calculating these values up to $N=800$ (higher values gave overflows
in double-precision floats), one obtains from \eqref{alpha:bound} the bounds
\begin{equation}
    0.60225 < \frac 12+ \frac1{2H} f_{E,\alpha}^\ast(-1) < 0.60316
\end{equation}
for the asymptotic mean of the proportion of essential nodes.
\end{proof}

\begin{figure}
\includegraphics[width=0.49\textwidth]{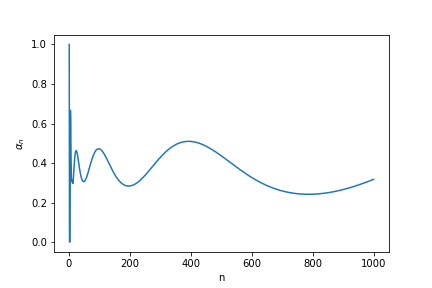}
\includegraphics[width=0.49\textwidth]{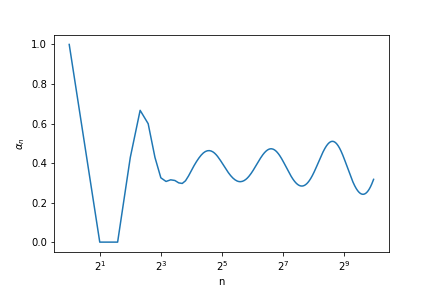}
\caption{The first values of $\alpha_n:=\E\varphi_\alpha(\ptrien)$ in the symmetric, binary case, on
a normal and a logarithmic $x$ scale, showing the oscillations.}
\label{fig:alpha}
\end{figure}

\bibliographystyle{plain}
\bibliography{master}
    
\end{document}